\documentclass[12pt]{amsart}
\usepackage{lineno,hyperref}
\usepackage[english]{babel}
\usepackage[utf8]{inputenc}
\usepackage{amsmath}
\usepackage{shuffle}
\usepackage{graphicx}
\usepackage{amsthm}
\usepackage{amsfonts}
\usepackage{hyperref}
\usepackage{upgreek}
\usepackage[ruled,vlined]{algorithm2e}
\usepackage{float}
\usepackage{color}
\usepackage{xcolor}
\usepackage[all]{xy}
\usepackage{subfig}
\usepackage{tabularx}
\usepackage{amssymb}
\usepackage{resizegather}
\usepackage{tikz-cd}
\usetikzlibrary{matrix, calc, arrows}
\usepackage{tikz}
\usetikzlibrary{positioning}
\usepackage{tkz-euclide}
\usepackage[foot]{amsaddr}
\usepackage{fullpage}
\usepackage{array}
\usepackage[T1]{fontenc}
\newtheorem{theorem}{Theorem}[section]
\newtheorem*{theorem*}{Theorem}
\newtheorem{lemma}[theorem]{Lemma}
\newtheorem*{lemma*}{Lemma}

\newtheorem{corollary}[theorem]{Corollary}

\newtheorem{proposition}[theorem]{Proposition}

\theoremstyle{definition}
\newtheorem{definition}[theorem]{Definition}
\newtheorem{example}[theorem]{Example}
\newtheorem{remark}[theorem]{Remark}

\usepackage{mathtools}
\def\multiset#1#2{\ensuremath{\left(\kern-.3em\left(\genfrac{}{}{0pt}{}{#1}{#2}\right)\kern-.3em\right)}}
\usepackage{bbold}

\newcommand{\poset}{\mathbf{Poset}}

\newcommand{\operad}{\mathcal{W}}
\newcommand{\spop}{\mathcal{SP}}

\newcommand{\chain}[1][n]{\langle #1\rangle}
\newcommand{\inv}[1]{{#1}^{-1}}
\newcommand{\nsomega}[2][n]{\Omega(#2,#1)}
\newcommand{\somega}[2][n]{\Omega^{\circ}(#2,#1)}

\newcommand{\zetaf}[1][P]{\mathfrak{Z}\left(#1\right)}
\newcommand{\zetafp}[1][P]{\mathfrak{Z}^+\left(#1\right)}
\newcommand{\zetafm}[1][P]{\mathfrak{Z}^{-}\left(#1\right)}

\providecommand{\keyword}[1]{\textbf{\textit{Keywords---}} #1}

\begin{document}
\title{An algebra over the operad of posets and structural binomial identities}

\author{José Antonio Arciniega-Nevárez$^1$}
\email{ja.arciniega@ugto.mx}
\address{$^1$División de Ingenierías, Universidad de Guanajuato, Guanajuato, México}

\author{Marko Berghoff$^2$}
\email{berghoffmj@gmail.com}
\address{$^2$Mathematical Institute, University of Oxford, Oxford, UK}

\author{Eric Rubiel 
Dolores-Cuenca$^3$}
\email{eric.rubiel@u.northwestern.edu}
\address{$^3$ Department of mathematics, Yonsei University, Seoul, Korea}

\maketitle

\begin{abstract} 
We study generating functions of strict and non-strict order polynomials of series-parallel posets, called order series. These order series are closely related to Ehrhart series and $h^*$-polynomials of the associated order polytopes. We explain how they can be understood as algebras over a certain operad of posets. 
Our main results are based on the fact that the order series of chains form a basis in the space of order series. This allows to reduce the search space of an algorithm that finds for a given power series $f(x)$, if possible, a poset $P$ such that $f(x)$ is the generating function of the order polynomial of $P$. In terms of Ehrhart theory of order polytopes, the coordinates with respect to this basis describe the number of (internal) simplices in the canonical triangulation of the order polytope of $P$. Furthermore, we derive a new proof of the reciprocity theorem of Stanley. As an application, we find new identities for binomial coefficients and for finite partitions that allow for empty sets, and we describe properties of the negative hypergeometric distribution.
\end{abstract}

\keyword{
Binomial Coefficient, Ehrhart series, Generating function, Negative Hypergeometric Distribution, Order Polynomial, Order Series, Partitions, Series Parallel Poset, Vandermonde Identity}

\section{Introduction}
For every poset ${P}$,  Stanley \cite{beginning} considered the problem of counting the numbers $\somega[n]{P}$ and $\nsomega[n]{P}$ of strict and non-strict order preserving maps from $P$ to the \emph{$n$--chain} $\chain[n]=\{1<\ldots<n\}$. 
In this paper we study the generating functions
\begin{equation*}
\zetaf[P]=\sum_{n=1}^\infty \somega[n]{P} x^n \quad \text{ and } \quad \zetafp[P]=\sum_{n=1}^\infty\nsomega[n]{P}x^n.
\end{equation*}
We call them the \emph{strict} and \emph{non-strict order series} of the poset $P$. 

These series are not completely new: Let $Poly(P)$ denote the order polytope of $P$ and $E_{Poly(P)}$ its Ehrhart polynomial. Using results of Stanley \cite{stanley:decompositions} and Macdonald \cite{M} we find
 \begin{align*}
     \zetafp[P] & = \sum_{n=1}^\infty \Omega(P,n) x^{n}  
 = x+\sum_{n=2}^\infty E_{Poly(P)}(n-1) x^n
     =x \frac{h^*(x)}{(1-x)^{|P|+1}}.
 \end{align*}
In other words, up to a degree shift we work with the Ehrhart series of $Poly(P)$, or its $h^*$-polynomial (aka $h^*$-vector) if we use the basis $\{\frac{1}{(1-x)^{i+1}}\}_{0\leq i\leq |P|}$ on the Ehrhart series of $Poly(P)$. 

For a $n$--chain we write $\zetaf[{\chain[n]}]=\zetaf[n]$ and $\zetafp[{\chain[n]}]=\zetafp[n]$. One finds (see Section \ref{sec:orderseries})
\begin{equation*}
\zetaf[n]= \frac{x^n}{(1-x)^{n+1}}\quad \text{and} \quad \zetafp[n]=\frac{x}{(1-x)^{n+1}}.
\end{equation*}

Computing the order series or, equivalently, the order polynomial of a poset is difficult in general. However, for certain families of posets, for instance series-parallel posets, it can be done algorithmically in polynomial time \cite{algo}. Note that it suffices to know one of the two order series, as this determines the other by a version of combinatorial reciprocity (see \cite{beginning} for order polynomials, Theorem \ref{thm:reci} and Proposition \ref{prop:equiv} for order series, and \cite{crt} in general). In the following we focus therefore on the strict order series $\zetaf$.

Posets have an operad structure, see \cite{OperadofPosets,Doppelgangers}. Here we consider an operad $\spop$ generated by concatenation and disjoint union of posets. Recall that series-parallel posets are built out of the singleton poset by iterating these two operations. Thus, series-parallel posets form an algebra over the operad $\spop$. We show below that the same holds for their order series (see Proposition~\ref{prop:osoperad}, Proposition~\ref{prop:operadmap} and Proposition~\ref{prop:mu+}). Furthermore, we demonstrate how chains are the basic building blocks not only for the set of series-parallel posets, but also for their order series (Section \ref{sec:orderseries}) as well as their order polytopes (Section \ref{subsec:EhrhartTheory}). 

The map that assigns to a poset its order series is not injective, that is, 
several posets can have the same order series. However, for some families of posets we can compute the preimages:  Given a power series $f(x)$ we describe in Section \ref{sec:algo} an algorithm to search for the posets $P$ that satisfy $f=\zetaf[P]$. 

Our main result is based on the observation that the (strict) order series of chains form a basis $\left\{\frac{x^i}{(1-x)^{i+1}}\right\}_{i \in \mathbb{N}}$ for the space of (strict) order series of series-parallel posets (Proposition \ref{prop:closed1}, see also Remark \ref{remark:reciprocity:coef}). This allows us to prove that

\begin{itemize}
\item The order series of every series-parallel poset $P$ is a finite sum of order series of chains, $\zetaf[P]=\sum c_i\zetaf[i]$. See Corollary~\ref{Cor:lc}.
\item The coefficients $c_i$ are non-negative integers. The indices $i$, such that $c_i \neq 0$, encode topological information of (the Hasse diagram of) $P$. See Proposition~\ref{prop:rep}.
\item The coefficients $c_i$ encode combinatorial information about the canonical triangulation of the order polytope of $P$. See Lemma~\ref{lemma:a}. In  Section~\ref{subsec:EhrhartTheory}, we compare our results with Ehrhart theory.
\item The existence of the operadic action on order series implies new binomial identities. See Section \ref{sec:comb}.
\end{itemize}

In \cite{D} Drinfeld describes a power series that corrects the notion of associativity. Motivated by the construction of Drinfeld, this paper started by the observation that order series can distinguish certain posets. For example, the two posets $\{a<b\}$ and $\{a,b\}$ differ by their order and their order series differ as well: $\frac{x^2}{(1-x)^3}\neq\frac{x}{(1-x)^2}+2\frac{x^2}{(1-x)^3}$.
\subsection{Related work}\label{section:rw}

We give a (short) account of closely related work.

Consider the category of finite sets, and let $\textbf{B}$ be the skeleton of the category whose objects are finite sets but the only morphisms are isomorphisms. The theory of $\textbf{B}$-species studies functors from $\textbf{B}$ to finite sets and power series associated to them \cite{species, grupoid}. A M\"obius-species \cite{mobius} is a functor from $\textbf{B}$ to the category of finite posets. In comparison, we study an operadic homomorphism from the algebra of series-parallel posets to power series. For example, the poset $\{ a<b<c,a<b'<c \}$ has the associated power series \begin{equation*}\frac{x^3}{(1-x)^4}+2\frac{x^4}{(1-x)^5}=x^3+6x^4+\cdots.\end{equation*} 
There is no $\textbf{B}$-species or M\"obius-species with the corresponding generating function, because a species only depends on the number of elements of the input set. We believe our series are a topological version of species. 

 The book \cite{eulerianB} introduces generating functions of order polynomials of posets to study partitions, Ehrhart polynomials and descent statistics. In contrast, we focus on the problem of computing the order series of a given poset. Moreover, we restrict our analysis to series-parallel posets since they have a rich algebraic structure (i.e., the operad structure referred to above, and explained below). In particular, all posets considered here are finite. 
 
A binomial poset $P$ is a locally finite poset with an element $\hat{0}$ so that $ \hat{0} \leq a$ for all $a\in P$, contains an infinite chain,  every interval $[s,t]$ is graded, and any two $n$-intervals contain the same number of maximal chains for any $n$ (see \cite{enumerative}).
For instance, the set $\mathbb{N}$ with the usual linear order is a binomial poset. To each binomial poset $P$ we can associate a subalgebra $R(P)$ of the incidence algebra of $P$: It consists of all functions $f$ such that $f(x,y)$ only depends on the length of the interval $[x,y]$. The algebra $R(P)$ is isomorphic to an algebra of generating functions with the usual product of functions. Stanley's construction can be used to study properties of invertible elements. In comparison, we study an algebra over the set-operad of series parallel posets. This formalism does not require an underlying vector space, and it is constructed from two operations that differ from the product of functions, the Hadamard product and the ordinal sum of power series which is a deformation of the usual product of functions. The generating functions that we study are algebraic of degree 1 according to \cite{enumerative2}.

In this article, a $n$-labeling of a poset $P$ is a map to $\chain[n]$ that preserves the order. This is a different notion to \cite{plabel} in which a $n$-labeling on a poset is defined as a poset $P$ and a map of sets. 

In regard to labelings of Hasse diagrams, there are many different variations of labelings of (directed) graphs. See for instance \cite{label}. However, we could not find a notion that is related to ours.

In regard to binomial identities, note that if we let $t=0$ in Proposition~\ref{Proposition:pc} we recover the Chu-Vandermonde identity. Most identities in \cite{lcomb} fix the top part of the combinatorial expression and vary the bottom part. For example, consider the generalized Vandermonde identity,
\begin{equation}\label{eqn:van}
{q_1+\cdots +q_p \choose n}=\sum_{n_1+\cdots n_p=n}\prod_{i=1}^p{q_i\choose n_i}.  
\end{equation}
Fixing a partition $q_1+\cdots +q_p$, the formula relates the binomial coefficient ${q_1+\cdots +q_p \choose n}$ with the product of binomial coefficients ${q_i \choose n_i}$ over all possible $p$-partitions $n=\sum n_i$.
In comparison, in Proposition~\ref{prop:new} we fix a partition $n_1+\cdots +n_p$ and relate the binomial coefficient ${q \choose n_1+\cdots +n_p}$ with the product of binomial coefficients ${q_i \choose n_i}$ where the terms $q_i$  depend on $q$ and the length $p$ of the partition. 

While the multivariate negative hypergeometric distribution is a well known topic in probability theory, as far as we are aware, the proofs in Section \ref{sec:prob} are new.

\section{Order series}\label{sec:orderseries}

Let $\poset$ denote the category of finite partially ordered sets whose morphisms are the strict order preserving maps. Let $\poset+$ denote the category with the same objects, but whose morphisms are non-strict order preserving (i.e., they satisfy $x<y \Rightarrow f(x)\leq f(y)$).

\begin{definition}
Define $\hom_{\poset+}(P,n)$ as the set of \emph{(non-strict) order preserving maps} from $P$ to $\chain$ and $\hom_{\poset}(P,n)$ as the set of \emph{strict order preserving maps} from $P$ to $\chain$. Define the order polynomials $\nsomega[n]
{P}$ and $\somega[n]{P}$ as the cardinality $|\hom_{\poset+}(P,n)|$ and $|\hom_{\poset}(P,n)|$, respectively. 
\end{definition}

For example $\somega[n]{\chain[m]}={n\choose m}$, while $\nsomega[n]{\chain[m]}=\multiset{n}{m}={n+m-1\choose m}$.

\begin{definition}
For a poset $Q$ we define a formal power series, called its \emph{order series}, by
\begin{equation*}
\zetaf[Q]=\zetaf[Q,x]=\sum_{n= 1}^\infty \somega[n]{Q}x^n.
\end{equation*}
\end{definition}
We have
\begin{gather}
    \zetaf[m]=
 \sum_{n=m}^\infty {n\choose m} x^n =\frac{x^m}{(1-x)^{m+1}}.
  \label{eqn:other1}
\end{gather} 
For a proof of the second equality see \cite[Equation~(1.5.5)]{gf}, \cite[Equation~(1.3)]{lcomb} or \cite[Equation~(1.3)]{eulerianB}.

If $x$ is the common variable for the power series $\zetaf[Q]$ and $\zetaf[W]$, we define 
\begin{equation}
\zetaf[Q]*\zetaf[W]=\zetaf[Q](1-x)\zetaf[W].\label{eq:*}  
\end{equation}
Note that here we are using the Cauchy product of power series.

We denote the \emph{Hadamard product} of power series by
\begin{equation*}\sum_{n=1}^\infty a_n x^n \sqcup \sum_{n=1}^\infty b_n x^n=\sum_{n=1}^\infty a_n b_n x^n.\end{equation*} We interpret the Hadamard product as disjoint union of power series.

The reason to work with power series instead of polynomials is that it highlights the structure of an underlying operad of posets (for an introduction to operads we refer to \cite{ope}).  Consider the operad $\spop$ generated by a binary associative and commutative operation $\sqcup$ and a binary associative operation $*$. The only unary operation is the identity. The action of $S_n$ is given by permutations on the inputs of the operations.

\begin{proposition}\label{prop:osoperad}
Order series are an algebra over the operad $\spop$. The above defined operations on power series
$*$ and $\sqcup$ are both associative, and $\sqcup$ is commutative.\end{proposition}

The set of series-parallel posets $SP$ is itself an algebra over the operad $\spop$.   Here the action of $*$ is the ordinal sum (or concatenation) of posets, and the action of $\sqcup$ is the disjoint union of posets.

For a finit posets $Q$ and $W$, the inclusions of $Q$ and $W$ in $Q*W$ allow us to see $Q$ and $W$ as a subposets of the concatenation. There is thus a function
\begin{equation}\label{eqn:surj1}
G_n:\bigcup_{l=1}^{n-1} 
\hom_{\poset}(Q,\chain[l])
\times \hom_{\poset}(W,\chain[n-l])\to \hom_{\poset}(Q*W,\chain[n]).
\end{equation}
On the term  $\hom_{\poset}(Q,\chain[l])
\times \hom_{\poset}(W,\chain[n-l])$ the function is given by 
$G_n(h,k)(v)=h(v)$, if $v\in Q$, and  $G_n(h,k)(v)=k(v)+l$, if $v\in W$. This function is surjective. Furthermore, for any element $f:Q*W\rightarrow \chain$ let $M$ be the maximum of $f|_{Q}$ and $m$ the minimum of $f|_{W}$ (actually $m=M+i$ for some integer $i\geq 0$). If $j\in\{0,1,\cdots,i-1\}$, we can partition $\chain$ in $\chain[M+j]$ and $\chain[n-M-j]$ in order to define $h_j=f|_Q:Q\to \chain[M+j]$, and $k_j:W\to \chain[n-M-j]$, given by $k_j(v)=f|_W(v)-M-j$.
See Figure~\ref{fig:multi1} for an example. 

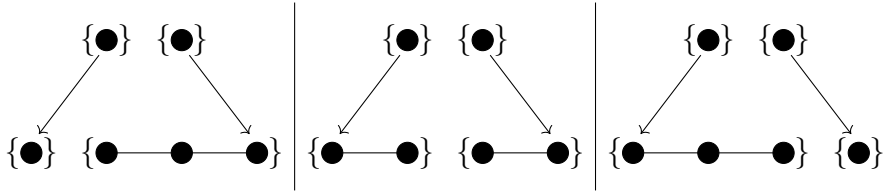
\begin{figure}[H]
\centering
\begin{tikzpicture}
\draw(3.5,-0.5) -- (3.5,2);
\draw(7.5,-0.5) -- (7.5,2);
\filldraw[black] (1,1.5) circle (4pt) node[anchor=east]{\{} node[anchor=west]{\}};
\draw[->] (0.9,1.3)--(0.1,0.25);
\draw[->] (6.1,1.3)--(6.9,0.25);
\filldraw[black] (2,1.5) circle (4pt) node[anchor=east]{\{} node[anchor=west]{\}};
\filldraw[black] (5,1.5) circle (4pt) node[anchor=east]{\{} node[anchor=west]{\}};
\draw[->] (4.9,1.3)--(4.1,0.25);
\draw[->] (10.1,1.3)--(10.9,0.25);
\filldraw[black] (6,1.5) circle (4pt) node[anchor=east]{\{} node[anchor=west]{\}};
\filldraw[black] (9,1.5) circle (4pt) node[anchor=east]{\{} node[anchor=west]{\}};
\draw[->] (8.9,1.3)--(8.1,0.25);
\draw[->] (2.1,1.3)--(2.9,0.25);
\filldraw[black] (10,1.5) circle (4pt) node[anchor=east]{\{} node[anchor=west]{\}};
\filldraw[black] (0,0) circle (4pt) node[anchor=east]{\{} node[anchor=west]{\}};
\filldraw[black] (1,0) circle (4pt) node[anchor=east]{\{};
\draw(1,0) -- (2,0);
\filldraw[black] (2,0) circle (4pt);
\draw(2,0) -- (3,0);
\filldraw[black] (3,0) circle (4pt) node[anchor=west]{\}};
\filldraw[black] (4,0) circle (4pt) node[anchor=east]{\{};
\draw(4,0) -- (5,0);
\filldraw[black] (5,0) circle (4pt)node[anchor=west]{\}};
\filldraw[black] (6,0) circle (4pt) node[anchor=east]{\{};
\draw(6,0) -- (7,0);
\filldraw[black] (7,0) circle (4pt)node[anchor=west]{\}};
\filldraw[black] (8,0) circle (4pt) node[anchor=east]{\{};
\draw(8,0) -- (9,0);
\filldraw[black] (9,0) circle (4pt);
\draw(9,0) -- (10,0);
\filldraw[black] (10,0) circle (4pt)node[anchor=west]{\}};
\filldraw[black] (11,0) circle (4pt)node[anchor=east]{\{}node[anchor=west]{\}};
\end{tikzpicture}
\caption{The function from $\langle 2\rangle$ to $\langle 4\rangle $ that sends $1\mapsto 1,$ and $2\mapsto 4$ can be split in 3 different ways: $\{\langle 1\rangle\rightarrow\langle 1\rangle\}\times\{ \langle 1\rangle\rightarrow \langle 3\rangle\}, \{\langle 1\rangle\rightarrow\langle 2\rangle\}\times\{ \langle 1\rangle\rightarrow \langle 2\rangle\} $ and $\{\langle 1\rangle\rightarrow\langle 3\rangle\}\times\{ \langle 1\rangle\rightarrow \langle 1\rangle$\}.}\label{fig:multi1}
 \end{figure}

For an integer $n>1$ define  \begin{equation*}
\triangle
\chain[n]=\bigsqcup_{i=1}^{n} (\chain[i],\chain[n-i]).
\end{equation*}
Also define $ \hom_{\poset+}((X,Y),\triangle\chain[n]) $ as the set of pairs of non strict order preserving maps  $(k:X\to\chain[i],h:Y\to\chain[n-i]), 1\leq i\leq n$. 
Furthermore, a computation shows
\begin{equation}
    \sum_{i=1} \somega[i]{X}x^i \sum_{j=1} \somega[j]{Y}x^j = \sum_{n=2}\somega[{\triangle\chain[n]}]{(X,Y)} x^n. 
\end{equation}

We will need the following auxiliary lemma.
\begin{lemma}\label{Lemma:ExactSequence1}
Assume that there is a sequence of sets
\begin{equation*}
\xymatrix{A\ar[r]^i& B\ar[r]^p &C\ar@/^/[l]^s,}
\end{equation*}
where 
\begin{itemize}
\item $i:A\rightarrow B$ is injective, 
\item $p:B\rightarrow C$ is surjective,
\item there is a split $s:C\to B$, such that $p\circ s=id$ and $s(C)\cap i(A)=\varnothing$, and
\item for every point $c\in C, s(c)\sqcup( i(A)\cap p^{-1}(c))=p^{-1}(c)$.
\end{itemize}

Then $s(C)\sqcup i(A)=B $.
\end{lemma}

 A sequence satisfying the previous lemma is called a \emph{splitting sequence of sets.}

\begin{proposition}\label{aux1}
Let $Q$ and $W$ be finite posets. Then
\begin{equation*}\label{eq:exact-seq-1}
\xymatrix{  \hom_{\poset}((Q,W),\triangle\chain[n-1])  
\ar[r]& \hom_{\poset}((Q,W),\triangle\chain[n])\ar[r]&\hom_{\poset}(Q*W,\chain[n])}
\end{equation*}
is a splitting sequence of sets.
\end{proposition}

\begin{proof}
For $k\in \hom_{\poset}(W,\chain[n-1])$ define $(k+1)\in \hom_{\poset}(W,\chain[n])$ by $(k+1)(v)=k(v)+1$. The shifting function
\begin{equation*}
Sh \colon \bigcup_{a+b=n-1} \hom_{\poset}(Q,\chain[a])
\times \hom_{\poset}(W,\chain[b]) \to \bigcup_{a+(b+1)=n}
\hom_{\poset}(Q,\chain[a])
\times \hom_{\poset}(W,\chain[b+1]),
\end{equation*}
given by sending $(h,k)$ into $(h, k+1)$ is injective. Let $G_n$ be  the function defined in \eqref{eqn:surj1}, and let $G_n^{-1}(f)$ be the  preimage of $f\in \hom_{\poset}(Q*W,\chain[n]) $. 
Then 
\begin{equation*}
G_n^{-1}(f)- Sh\left(\bigcup_{a+b=n-1}
\hom_{\poset}(Q,\chain[a])
\times
\hom_{\poset}(W,\chain[b])
\right)
\end{equation*}
has only one element, namely the pair $(h,k)$ where $h=f|_Q:Q\rightarrow \chain[m]$ and $k:W\rightarrow \chain[n-m+1]$ is given by $k(v)=f|_W(v)-m+1$ (again, $m$ is the minimum of $f|_W$). 

Choosing this unique point we define a section 
\begin{equation*}
s:\hom_{\poset}(Q*W,\chain[n])\to \bigcup_{c+d=n} \hom_{\poset}(Q,\chain[c])\times \hom_{\poset}(W,\chain[d]),
\end{equation*}
whose image is disjoint from the image of the shifting function $Sh$. This holds because $k(m)=1$ but no function of the form $g+1$ has $1$ in the codomain. 
Then the hypothesis of Lemma \ref{Lemma:ExactSequence1} holds. 
\end{proof}

Suppose that there is a short splitting sequence of finite sets
 \begin{equation}\label{eq:exact-seq}
 \xymatrix{
O_{(X,Y), \triangle\langle n-i_2\rangle}\ar[r]^<<<<{\iota}&O_{(X,Y), \triangle\langle n-i_1\rangle}\ar[r]^{\pi}&O_{X*Y,n}}
\end{equation}
with $i_1$ and $i_2$ non-negative integers and $n\in\mathbb{N}$, then
\begin{equation}\label{eq:product-cauchy}
\begin{aligned}
\sum_{n=2} \somega[n]{\mu(Y,W)}x^n=&x^{ i_1}\sum_{n=2} \somega[\triangle{\chain[n- i_1]}]{(Y,W)}x^{n- i_1}-\\
&-x^{ i_2}\sum_{n=2} \somega[\triangle{\chain[n- i_2]}]{(Y,W)}x^{n- i_2}\\
=&x^{ i_1}\sum_{j=1}\somega[j]{Y}x^j \sum_{k=1}\somega[k] {W}x^k\\
 &- x^{ i_2} \sum_{j=1}\somega[j]{Y}x^j \sum_{k=1}\somega[k]{W}x^k.
\end{aligned}
\end{equation}

\begin{proposition}\label{prop:operadmap}
There is a  homomorphism $SP\to \{\text{order series}\}$ of algebras over the operad $\spop$. It maps a poset $P$ to the series $\zetaf[P]$ and satisfies
 \begin{equation}
\zetaf[Q*W]=\zetaf[Q]*\zetaf[W],  \label{eqn:concat}\end{equation}
\begin{equation}
\zetaf[Q\sqcup W]=\zetaf[Q]\sqcup \zetaf[W].\label{eqn:Hadamard1}
\end{equation}
\end{proposition}
\begin{proof}
Equation~\eqref{eqn:Hadamard1} follows by counting elements in $\hom_{\poset}(P\sqcup Q,n)$.  
Equation \eqref{eqn:concat} follows by Equation~\ref{eq:product-cauchy} with $i_2=1, i_1=0$. 
\end{proof}

Both relations are well-known at the level of order polynomials, (see for example \cite{algo}).

The following statement allows to make explicit computations.
\begin{proposition}\label{prop:closed1}
For $0\leq m, k :$ 
\begin{equation*}
\zetaf[k]\sqcup \zetaf[m]=\frac{x^k}{k!}\frac{d^k}{dx^k}\zetaf[m]=\sum_{n=0}^k{m+n \choose k}{k \choose n}\zetaf[m+n]. 
\end{equation*}
\end{proposition}

\begin{proof}
First note that 
\begin{equation}\label{eqn:principal}
\zetaf[m]=\sum_{n}^\infty{n\choose m}x^n
=\sum_{n}^\infty \frac{x^m}{m!}\frac{d^m}{dx^m}x^n.
\end{equation}
From \eqref{eqn:principal} we obtain  
\begin{align*}
\zetaf[k]\sqcup\zetaf[m]
&=\sum_n^\infty {n\choose k}{n\choose m}x^n\\
&=\sum_n^\infty\frac{x^k}{k!}\frac{d^k}{dx^k} {n\choose m}x^n\\
&=\frac{x^k}{k!}\frac{d^k}{dx^k}\zetaf[m],
\end{align*}
and therefore
\begin{align*}
\zetaf[k]\sqcup \zetaf[m]
&=\frac{x^k}{k!}\frac{d^k}{dx^k}\zetaf[m]\nonumber\\
&=\frac{x^k}{k!}\frac{d^k}{dx^k}\frac{x^m}{(1-x)^{m+1}}\nonumber\\
&=\frac{x^k}{k!}\sum_{n=0}^k{k\choose n}\left(\frac{d^{k-n}}{dx^{k-n}}x^m\right)\left(\frac{d^{n}}{dx^{n}}(1-x)^{-m-1}\right) \nonumber\\
&=\sum_{n=0}^k{k\choose n}\frac{m!}{(m+n-k)!k!}\frac{(m+n)!}{m!}x^{m+n}(1-x)^{-m-n-1} \nonumber\\
&=\sum_{n=0}^k{m+n \choose k}{k \choose n}\zetaf[m+n]. 
\end{align*}
\end{proof}
Another proof of Proposition~\ref{prop:closed1}, can be deduced from \cite[(6.44)]{lcomb}.

\begin{corollary}
The strict order series of any series-parallel poset is a linear combination of order series of chains.\label{Cor:lc}
\end{corollary}
\begin{proof} Series-parallel posets are generated by the singleton poset under the operations $*$ and $\sqcup$. Therefore their order series are generated by the chain $\zetaf[1]=\frac{x}{(1-x)^2}$ under the operations $*$ and $\sqcup$. These operations are linear on power series and send chains into linear combinations of chains. 
\end{proof}

\begin{example}\label{example:1}
 Let $D$ be the unary operation on posets defined by
\begin{equation*}
D(P)=\chain[1]*\big(\chain[1]\sqcup P\big)*\chain[1].
\end{equation*}
If the input of $D$ is an $n$-chain, we compute:
\begin{eqnarray*}
\zetaf[D({\chain[n]})]&=&\zetaf[{1}]*(\zetaf[{\chain[1]\sqcup \chain[n]}])*\zetaf[{1}]
\\
&=&\frac{x}{(1-x)^2}*(n\frac{x^n}{(1-x)^{n+1}}+(n+1)\frac{x^{n+1}}{(1-x)^{n+2}})*\frac{x}{(1-x)^2}\\
&=&n\frac{x^{n+2}}{(1-x)^{n+3}}+(n+1)\frac{x^{n+3}}{(1-x)^{n+4}}\\
&=&n\zetaf[{n+2}]+(n+1)\zetaf[{n+3}].
\end{eqnarray*}
\end{example}

\begin{remark}
From the explicit description as a finite sum $\zetaf[P]=\sum a_i\zetaf[i]$ the $n$-order polynomial can be extracted as $\somega[n]{P}=\sum a_i{i\choose n}$.
\end{remark}

\begin{definition}
Recall the definition of non-strict order series,
\begin{equation*}
\zetafp[Q]=\sum_{n=1}^\infty \nsomega[n]{Q} x^n.
\end{equation*}
\end{definition}

Define the concatenation by \begin{equation*}
\zetafp[Q]*^+\zetafp[W]=\zetafp[Q]\frac{1-x}{x}\zetafp[W].
\end{equation*}
and note that the non-strict order series of a disjoint union of posets is still given by the Hadamard product of their respective non-strict order series.

\begin{example}
We have $\zetafp[i]=\frac{x}{(1-x)^{i+1}}=\sum {\multiset{n}{i}}x^n$ where the last equality follows from Equation~\eqref{eqn:other1}.
 Now consider the poset $D(\chain[n])$ for the operation $D$ defined in Example \ref{example:1} above. Then
 \begin{align*}
  \zetafp[{D(\chain[n])}]= & \zetafp[1] *^+ \left( \zetafp[n] \sqcup \zetafp[1] \right) *^+ \zetafp[1]   \\
  = &  \frac{x}{(1-x)^2}\frac{d\zetafp[n]}{dx}
  =  \frac{nx^2+x}{(1-x)^{n+4}},
 \end{align*}
 which can be calculated by hand, but also follows from the result in Example \ref{example:1} by Theorem \ref{thm:reci} below.
\end{example}

\begin{proposition}\label{prop:mu+} Non-strict order series with $*^+,\sqcup$ are an algebra over the operad $\spop$, and there is a $\spop$ homomorphism $SP\rightarrow \{\text{non strict order series}\}$, that is,
\begin{equation}
\zetafp[Q*W]=\zetafp[Q]*^+\zetafp[W],\label{hom+a}
\end{equation}
\begin{equation}
\zetafp[Q\sqcup W]=\zetafp[Q]\sqcup\zetafp[W].\label{hom+u}
\end{equation}
\end{proposition}
\begin{proof}

 Given a pair $(i,j)$, if $n=i+j$, then there is a function \begin{equation*}
\hom_{\poset+}((X,Y),(\chain[i],\chain[j]))\to \hom_{\poset+}({X*Y, \chain[n-1]})
\end{equation*}
that takes the pair of order preserving maps $h:X\mapsto\chain[i]$ and $k:Y\mapsto\chain[j]$ and defines  $F:X*Y\to\chain[n-1]$ by $F(v)=h(v)$ if $v\in X$ and $F(v)=k(v)+i-1$ if $v\in Y$.
The sequence 
\begin{equation*}\label{eq:exact-seq-2}
\xymatrix{ \hom_{\poset+}((X,Y),\triangle\chain[n-1])\to  \hom_{\poset+}((X,Y),\triangle\chain[n])\to \hom_{\poset+}({X*Y, \chain[n-1]})}
\end{equation*}
satisfies Lemma~\ref{Lemma:ExactSequence1}. The relation \eqref{hom+a} then follows by Equation~\eqref{eq:product-cauchy} with $i_2=0,i_1=-1$. Equation~\eqref{hom+u} follows by the same arguments as in the $\poset$ case.
\end{proof}

At the level of posets, we have a faithful functor between the categories $\poset\rightarrow \poset+$ that fixes the objects and for each $P,Q$ injects $\hom_{\poset}(P,Q)$ into $\hom_{\poset+}(P,Q)$. At the level of order series, there is an operadic isomorphism between strict and non-strict order series.   

Let $i$ denote the map $x\mapsto x^{-1}$ and define a map $\iota$, the reciprocity morphism, as follows. If $f$ is the analytic continuation of a strict or non-strict order series of a poset $P$, then 
\begin{equation*}\iota(f):= (-1)^{|P|+1} \big(f \circ i\big). \end{equation*}

\begin{theorem}\label{thm:reci}
For any series-parallel poset $P$ we have 
\begin{equation*}
\iota(\zetaf)=\zetafp \quad\text{and}\quad \iota^2 = \mathrm{id}.
\end{equation*}
In other words, $\zetaf$ is the expansion of its analytic continuation at 0, while $\zetafp$ is the expansion at $\infty$ of the \emph{same} function (up to a sign). 
\end{theorem}

\begin{proof} By definition it follows that $\iota^2=\mathrm{id}$. We will show that $\iota$ commutes with concatenation and disjoint union of chains, that is, it is a morphism between two $\spop$-algebras. Since these operations generate every series-parallel poset, this proves the theorem. 

Since $(1-x^{-1}) =-\frac{(1-x)}{x}$, we see that $\iota$ commutes with concatenation, 
\begin{eqnarray*}
\iota (\zetaf[n] * \zetaf[m])&=&\iota( \frac{x^n}{(1-x)^{n+1}} (1-x) \frac{x^m}{(1-x)^{m+1}}) \\ &=&(-1)^{n+m+1}(-1)^{n+1} \frac{x}{(1-x)^{n+1}} (-1)\frac{(1-x)}{x} (-1)^{m+1}\frac{x}{(1-x)^{m+1}} \\
&=& \iota \zetaf[n] *^+ \iota \zetaf[m].
\end{eqnarray*}

We claim 
  \begin{equation*} \iota\zetafp[n] \sqcup \iota\zetafp[m]=
 \iota \big(\zetafp[n]\sqcup \zetafp[m]\big).
 \end{equation*}
Since $\iota^2=\mathrm{id}$ this readily also implies  
\begin{equation*} \iota\big(\zetaf[n] \sqcup \zetaf[m]\big) =
 \iota^2 \left( \zetafp[n]\sqcup \zetafp[m] \right)=\iota\zetaf[n]\sqcup \iota\zetaf[m].
 \end{equation*}

Abbreviate $\zetafp[m]$ by $f$. Using Lemma \ref{lemma:derivation} below we calculate 
\begin{align*}
\iota\zetafp[n] \sqcup \iota(f) &  =\zetaf[n]\sqcup \iota(f) =  (-1)^{m+1}\frac{x^n}{n!}d_n(f \circ i) \\
& =  (-1)^{m+n+1}\sum_{k=1}^n  \binom{n-1}{k-1}\frac{x^{-k}}{k!}  d_kf(\inv{x}) =  \iota \left( \zetaf[n] \sqcup f \right) .
\end{align*}
\end{proof}

\begin{lemma}
Let $f$ be $n$-times differentiable and $x\neq 0$. Then  
\begin{equation*}
\frac{x^n}{n!}\frac{d^n}{dx^n} (f \circ i)
=(-1)^n \sum_{k=1}^n \binom{n-1}{k-1} \frac{x^{-k}}{k!}  \frac{d^{k}f}{dx^k}(\inv{x}).
\end{equation*}\label{lemma:derivation}
\end{lemma}
\begin{proof}

Let us abbreviate $\frac{d^n}{dx^n}$ by $d_n$. Applying the \emph{Fa\`{a} di Bruno formula} for the chain rule of higher derivatives we get
\begin{equation*}
d_n(f \circ i)=\sum_{k=1}^n d_kf(\inv{x}) B_{nk}(d_1i,d_2i, \ldots, d_{n-k+1}i).
\end{equation*}
Here the \emph{(incomplete) Bell polynomials} $B_{nk}$ are defined by
\begin{equation*} 
B_{nk}(x_1,\ldots, x_{n-k+1})= \sum \frac{n!}{j_1!j_2!\cdots j_{n-k+1}!}
\left(\frac{x_1}{1!}\right)^{j_1}\cdots\left(\frac{x_{n-k+1}}{(n-k+1)!}\right)^{j_{n-k+1}},
\end{equation*}
where the sum is taken over all sequences $j_1,\ldots , j_{n-k+1}$ of non-negative integers such that 
\begin{equation*} j_1 + j_2 + \cdots + j_{n-k+1} = k, \quad 
j_1 + 2 j_2 + 3 j_3 + \cdots + (n-k+1)j_{n-k+1} = n.
\end{equation*}
Now we compute $B_{nk}(d_1i, d_2i,\ldots, d_{n-k+1}i)$ to be given by
\begin{align*}
   & \sum  \frac{n!}{j_1!j_2!\cdots j_{n-k+1}!}
\left(\frac{-x^{-2}}{1!}\right)^{j_1}\cdots\left(\frac{(-1)^{n-k+1}(n-k+1)!x^{-n+k-2}}{(n-k+1)!}\right)^{j_{n-k+1}} \\
= & \sum  \frac{n!}{j_1!j_2!\cdots j_{n-k+1}!}
(-1)^{\sum_{i=1}^{n-k+1} i j_i} x^{-\sum_{i=1}^{n-k+1} i j_i} x^{-\sum_{i=1}^{n-k+1}  j_i} \\
= & \sum  \frac{n!}{j_1!j_2!\cdots j_{n-k+1}!}
(-1)^{n} x^{-n} x^{-k} \\
= &(-1)^{n} x^{-n} x^{-k} B_{nk} (1!, \ldots, (n-k+1)!) \\
= & (-1)^{n} x^{-n} x^{-k}\binom{n-1}{k-1}\frac{n!}{k!},
\end{align*}
where the last identity is a well-known property of the incomplete Bell polynomials \cite{comtet}.
\end{proof}

Our reciprocity theorem is equivalent to Stanley's version \cite{beginning}.
\begin{proposition}\label{prop:equiv}
For series-parallel posets Theorem \ref{thm:reci} is equivalent to the reciprocity theorem of Stanley.
\end{proposition}
\begin{proof}
 See the proof of \cite[Lemma 3.15.11]{enumerative}.
\end{proof}

\begin{remark}\label{remark:reciprocity:coef}
By reciprocity we have the following relation among the coefficients of order series  \begin{equation*}
\zetafp[{P}]=\sum_{i=j_0}^{n} (-1)^{n-i}a_i\zetafp[i] \quad \Longleftrightarrow \quad \zetaf[{P}]=\sum_{i=j_0}^{n} a_i\zetaf[i].
\end{equation*}
\end{remark}

\begin{remark}\label{rem:moebius}
The singularities of $\zetaf$ and $\zetafp$ at 1 are simple poles. It follows that the Moebius transformation $x\mapsto \frac{1}{1-x}$, sending 1 to $\infty$, turns the analytic continuations of order series into polynomials. For example, for chains we obtain
\begin{equation*}
\zetaf[n]|_{y=\frac{1}{1-x}}= \sum_{k=0}^n \binom{n}{k} (-1)^k y^{n-k+1}, \quad \zetafp[n]|_{y=\frac{1}{1-x}}= y^{n+1}-y^n.
\end{equation*}
\end{remark}

Since every order series of a series-parallel poset is a linear combination of order series of chains, there are new algebraic relations between order series that are not present on the level of posets. For example, the order series of the poset $\chain[s]\sqcup\chain[n]$ becomes a linear combination of order series of chains while the poset is not a chain.  The following lemma describes the interaction between concatenation and disjoint union of order series and implies a new binomial identity as explained in Proposition~\ref{Proposition:pc}.

\begin{lemma}
Let $s>0$. The following relationship holds for power series $f$ and $g$:
\begin{multline}\label{compatibility}
\zetaf[s]\sqcup (f*g)=\sum_{a+c=s} (\zetaf[a]\sqcup f)*(\zetaf[c]\sqcup g)-\\\left(\sum_{a+c=s-1} (\zetaf[a]\sqcup f)*(\zetaf[c]\sqcup g)\right) *\zetaf[1].
\end{multline}
\end{lemma}
\begin{proof}
We use the notation of ${s\choose a,b,c}=\frac{s!}{a!b!c!}$ for the multinomial coefficient.

\begin{align*}
\zetaf[s]\sqcup (f*g)
&=\frac{x^s}{s!}\frac{d^s}{d x^s}(f(1-x)g)\nonumber\\
&=\frac{x^s}{s!}\sum_{a+b+c=s} {s\choose a,b,c}\frac{d^a}{d x^a} f\frac{d^b}{d x^b} (1-x)\frac{d^c}{d x^c} g \\
&=\frac{x^s}{s!}(\sum_{a+c=s} {s\choose a,0,c}\frac{d^a}{d x^a} f (1-x)\frac{d^c}{d x^c} g\\
&+\sum_{a+c=s-1} {s\choose a,1,c}\frac{d^a}{d x^a} f\frac{d}{d x} (1-x)\frac{d^c}{d x^c} g) \\
&=\sum_{a+c=s} (\zetaf[a]\sqcup f)*(\zetaf[c]\sqcup g)\\&-\frac{x^{s-1}}{s-1!}\sum_{a+c=s-1} {s-1\choose a,c}\frac{d^a}{d x^a} f\frac{d^c}{d x^c} g (1-x)^2\frac{x}{(1-x)^2}\\
&=\sum_{a+c=s} (\zetaf[a]\sqcup f)*(\zetaf[c]\sqcup g)\\&-\left(\sum_{a+c=s-1} (\zetaf[a]\sqcup f)*(\zetaf[c]\sqcup g)\right) *\zetaf[1].
\end{align*}
\end{proof}

\section{Representability of posets}\label{Sec:rep}
In this section we introduce a family of posets in which the coefficients $a_i$ of the order series $\zetaf[P]=\sum a_i\zetaf[i]$ encode topological information. Our results allow to reduce the search space for an algorithm that, given a power series $f(x)$, if possible, finds  a poset $P$ such that $f=\zetaf[P]$.

Consider the unary operation $D$ defined in Example~\ref{example:1},
\begin{equation*}
D(P)=\chain[1]*\big(\chain[1]\sqcup P\big)*\chain[1].
\end{equation*}
This adds a new maximum $x_1$ and a new minimum $x_0$ to $P$, and a third element $y$ with $x_0 < y <x_1$. Geometrically we are attaching a handle (with a point) to the Hasse diagram of $P$. 

Let $\operad$ be the free operad generated by $\mu$ and $D$. We are interested in the $\operad$--algebra generated by the poset with a single element where $\mu$ is concatenation and $D$ is attaching a handle. We call elements of this $\operad$--algebra \emph{Wixárika posets}. For an example see Figure~\ref{fig:melted}, which depicts the Hasse diagram of a Wixárika poset (ordered from left to right). 
 The name is due to the resemblance of the Hasse diagrams of such posets to the intricate colorful necklaces crafted by the Wixárika community in North America.

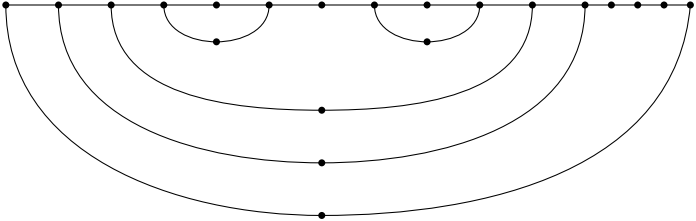
\begin{figure}[htb]
\centering
\begin{tikzpicture}[scale=0.7]
\coordinate  (v1) at (-7,0);  
\coordinate  (v2) at (-6,0);
\coordinate  (v3) at (-5,0);
\coordinate  (v4) at (-4,0);
\coordinate  (v5) at (-3,0);
\coordinate  (v6) at (-2,0);
\coordinate  (v7) at (-1,0);
\coordinate  (v8) at (0,0);
\coordinate  (v9) at (1,0);
\coordinate  (v10) at (2,0);
\coordinate  (v11) at (3,0);
\coordinate  (v12) at (4,0);
\coordinate  (v13) at (4.5,0);
\coordinate  (v14) at (5,0);
\coordinate  (v15) at (5.5,0);
\coordinate  (v16) at (6,0);
\coordinate (w1) at (-3,-.7);
 \coordinate  (w2) at (1,-.7);  
 \coordinate  (x1) at (-1,-2);
  \coordinate  (y1) at (-1,-3);
  \coordinate  (z1) at (-1,-4);
  
      \draw (v1) -- (v16);
\draw (v4) to[out=-90, in=180] (w1); \draw (w1) to[out=0, in=-90] (v6);
\draw (v8) to[out=-90, in=180] (w2); \draw (w2) to[out=0, in=-90] (v10);
\draw (v3) to[out=-90, in=180] (x1); \draw (x1) to[out=0, in=-90] (v11);
\draw (v2) to[out=-90, in=180] (y1); \draw (y1) to[out=0, in=-90] (v12);
\draw (v1) to[out=-90, in=180] (z1); \draw (z1) to[out=0, in=-95] (v16);

\fill[black] (v1) circle (.0666cm);
\fill[black] (v2) circle (.0666cm);
\fill[black] (v3) circle (.0666cm);
\fill[black] (v4) circle (.0666cm);
\fill[black] (v5) circle (.0666cm);
\fill[black] (v6) circle (.0666cm);
\fill[black] (v7) circle (.0666cm);
\fill[black] (v8) circle (.0666cm);
\fill[black] (v9) circle (.0666cm);
\fill[black] (v10) circle (.0666cm);
\fill[black] (v11) circle (.0666cm);
\fill[black] (v12) circle (.0666cm);
\fill[black] (v13) circle (.0666cm);
\fill[black] (v14) circle (.0666cm);
\fill[black] (v15) circle (.0666cm);
\fill[black] (v16) circle (.0666cm);
\fill[black] (w1) circle (.0666cm);
\fill[black] (w2) circle (.0666cm);
\fill[black] (x1) circle (.0666cm);
\fill[black] (y1) circle (.0666cm);
\fill[black] (z1) circle (.0666cm);
\end{tikzpicture}
\caption{Hasse diagram of a Wixarika poset. 
}\label{fig:melted}
\end{figure}

We can describe elements in $\operad$ by words in the letters $\mu$ and $D$. Let $w$ be such an element. By abuse of notation, we write $w(\chain[1])$ to represent the substitution of the poset $\chain[1]$ on all leaves of the tree describing the word $w$ see Figure~\ref{fig:tree}.

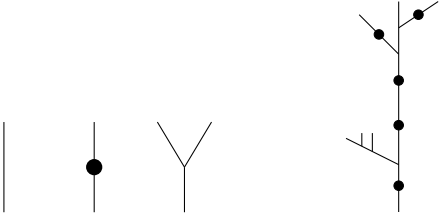
\begin{figure}[htb]
\centering{
\begin{tikzpicture}[scale=.6]
\coordinate  (v1) at (-5,1);  
\coordinate  (v2) at (-5,-1);
      \draw (v1) -- (v2);
      
\coordinate  (w1) at (-3,1);
\coordinate  (w2) at (-3,-1);
\coordinate  (w3) at (-3,0);   
\draw (w1) -- (w2);
\fill[black] (w3) circle (.18cm);

\coordinate  (x1) at (-1.6,1);
\coordinate  (x2) at (-.4,1);
\coordinate  (x3) at (-1,0);
\coordinate  (x4) at (-1,-1);
     \draw (x1) -- (x3);
      \draw (x2) -- (x3);
      \draw (x3) -- (x4);
\end{tikzpicture}
\hspace{1.5cm}
 \begin{tikzpicture}[scale=.35]
 \draw (0,0) -- (0,8);
 \filldraw[black] (0,1) circle (.18cm);
  \filldraw[black] (0,3.3) circle (.18cm);
   \filldraw[black] (0,5) circle (.18cm);
   
    \draw (0,6) -- (-1.5,7.5);
   \filldraw[black] (-.75,6.75) circle (.18cm);
   
       \draw (0,7) -- (1.5,8);
   \filldraw[black] (.75,7.5) circle (.18cm);
   
      \draw (0,1.8) -- (-2,2.8);
      \draw (-1.4,2.5) -- (-1.4,3);
      \draw (-1,2.3) -- (-1,3);
\end{tikzpicture}}
\caption{Left: The identity operation, the operation $D$ and the operation $\mu$. Right: A tree representing the element of $\operad$ in Figure~\ref{fig:melted}.}\label{fig:tree}
\end{figure}

In \cite{algo}, an algorithm is described to compute the order polynomials. The same algorithm can be adapted to our setting to compute order series, and in fact, from the point of view of operads, the algorithm amounts to evaluate the operations on the tree from top to bottom.  Similarly, we write $w(\zetaf[1])$  
to represent the substitution of the corresponding elements in power series. 
For example, the tree of Figure~\ref{fig:tree} corresponds to $w=D(\_*(\_*(\_*D(D(D(\_)*(\_*D(\_)))))))$. Thus, following Example~\ref{example:1}, we compute
$w(\zetaf[1])=882\zetaf[16]+7995\zetaf[17]+27232\zetaf[18]+43792\zetaf[19]+33552\zetaf[20]+9880\zetaf[21]$.

We will say that a power series $f(x)$ is \emph{represented} if there is a series-parallel poset $P$ such that $f=\zetaf[P]$. In order to determine which power series can be represented, we shall determine topological invariants of a poset $P$ from its associated power series $\zetaf[P]$, such as the Betti number of $P$ (which is short for \emph{the first Betti number of the Hasse diagram of $P$}). With that information we describe an algorithm which determines if $f(x)$ can be represented and if so, it explicitly constructs all possible series-parallel posets $P$ so that $f=\zetaf[P]$. Posets that have the same order series are called \emph{Doppelg\"{a}nger} posets.

\begin{proposition} \label{prop:rep}
Consider $w\in \operad$,  a word in the characters $*$ and $D$. Let  $f(x)=w(\zetaf[1])$ and $P=w(\chain[1])$.

\begin{enumerate}
\item The series $f(x)$ admits a unique description 
\begin{equation*}
f(x)= a_i\zetaf[i]+a_{i+1}\zetaf[i+1]+\cdots+a_{k}\zetaf[k],
\end{equation*}
where the coefficients $a_j$ are non-negative integers. 
\item The number $i$ is the number of elements in a maximal chain in $P$.\label{prop:point}
\item The number $k$ is the number of elements in $P$.
\item The difference $d=k-i$ is the Betti number of $P$, equal to the number of times the character $D$ occurs in $w$.
\item The difference $m=i-2d-1$ is the number of times the character $*$ occurs in $w$.
\item The number $m+1$ is the number of leaves in the tree $w$.
\item\label{item:cond} $\sum_{u=1}^{k}(-1)^{k-u}a_u=1$.
\item The first coefficient $a_i$ admits a factorization  $a_i=\prod_{j=1}^d t_j$ with the following property. For any $t_j\geq 1$ there are $r_{t_j},s_{t_j},$ and $u_{t_j}$ with $t_j<u_{t_j}$, so that when we evaluate the word $w$ on $\zetaf[1]$ then  $r_{t_j}\zetaf[t_j]+\cdots+s_{t_j}\zetaf[u_{t_j}]$ is an input of some $D$.
\end{enumerate}
\end{proposition}
\begin{proof}
We shall give a proof of each item.
\begin{enumerate}
\item Follows from Proposition~\ref{prop:closed1} and Equation~\eqref{eq:*}.

\item  If $*$ appears in  $w$, it concatenates subposets of any maximal chain. If $D$ appears in $w$,  it adds a maximum and a minimum point to the maximal chain in the input poset, which is a subposet of any maximal chain. 

We conclude from  Proposition~\ref{prop:closed1} and Equation~\eqref{eq:*} that when we evaluate $w$, the lowest term counts the number of points in a maximal chain.

\item We prove this in Lemma~\ref{lemma:gen} in a more general setting. 
\item
From Proposition~\ref{prop:closed1} it follows that we can track the number of times $D$ acts by comparing the parameter of the monomial with lowest degree and the monomial of highest degree. Geometrically, $D$ modifies the Hasse diagram by increasing the number of holes.
\item If the word $D$ occurs $d$ times in $w$, the only way to obtain the term $\zetaf[i]$ is if the word $*$ occurs $i-2d-1$ times in $w$.

\item Consider the representation of $w$ as a binary tree with some marked edges. If the tree has $m$ binary branches, then it has $m+1$ leaves.
\item Follows from Equation \eqref{eqn:Ehrhart} and the fact that $h_0^*=1$
\item Follows from  Proposition~\ref{prop:closed1}.
\end{enumerate}
    \end{proof}

\subsection{The inverse problem}\label{sec:algo}
In this section we study the following problem. When is a power series $f(x)=\sum_{i=i_0}^{i_j} f_i\zetaf[i], f_{i_0},f_{i_j}\neq 0$,  represented as $\zetaf[P]$ for some poset $P$? Two posets are called Doppelg\"{a}nger posets if they share the same order series, for example 
$\chain[1]*(\chain[1]\sqcup\chain[1]\sqcup\chain[1])$ and $\chain[2]\sqcup \chain[2]$ have the same order series:
$\zetaf[2]+6\zetaf[3]+6\zetaf[4]$ \cite{Doppelgangers}.

Note that if we order the chains in the decomposition $\zetaf[P]=a\zetaf[i]+\cdots + b\zetaf[j],\zetaf[Q]=a_2\zetaf[i_2]+\cdots + b_2\zetaf[j_2]$, the terms $a,i,b, j$ change under the rules:
\begin{equation}
\zetaf[P]*\zetaf[Q]=aa_2\zetaf[ i+i_2]+\cdots+bb_2\zetaf[j+j_2]  \label{eqn:ex1}
\end{equation}
and
\begin{equation}
D(\zetaf[P]) = ia\zetaf[i+2]+\cdots+ (j+1)b\zetaf[ j+3].\label{eqn:ex2}  
\end{equation}

If we have a candidate poset $P$,
to reduce the computational time, we first compute the lowest and highest order term of the factorization $\zetaf[P]=a\zetaf[i]+\cdots + b\zetaf[j],$ and compare it with $f_{i_0}, f_{i_j}$. To select candidate posets, we do not need to explore all possible words $w$. Since $*$ is associative, we can pick $(\_*(\_*\cdots(\_*\_))\cdots~)$ as a representative of every composition of $n$ operations $*$. We call this representative an \emph{$n$-corolla}. In other words, the $2$--corolla is $(\_*\_)$, the $3$--corolla is $(\_*(\_*\_))$, and so on. If a word $w$ is a solution to our problem, then any permutation of the inputs of the corollas will also be a solution to our problem. In our setting, some of those Doppelg\"{a}ngers are explained because the poset concatenation is not commutative, but the power series concatenation is commutative.

The language of operads allows us to choose representatives from equivalence classes of trees. Consider the two-colored symmetric operad 2--$\operad$. This operad has a binary operation $D$, colored in  red, and for every $n>1$ an $n$--ary operation, given by an $n$-corolla $*$, colored in green. Red colored operations have no restrictions, but green ones can only be composed with red ones. In other words, a green corolla cannot be grafted into a green corolla. 

\begin{proposition}\label{thm:alg}
Given a finite sum $f(x)=\sum_{i=i_0}^{i_j} f_i\zetaf[i], f_{i_0},f_{i_j}\neq 0$ where $f_{i_k},{0\leq k\leq j}$ are non-negative integers, there is an algorithm to determine if $f(x)=\zetaf[P]$ for some Wixarika poset $P$ and, in the positive case, the algorithm returns all possible posets representing $f(x)$.
\end{proposition}

\begin{proof}
Let $f(x)=\sum_{i=i_0}^{i_j} f_i\zetaf[i], f_{i_0},f_{i_j}\neq 0$. We first restrict the space of posets that can represent $f$. From Proposition~\ref{prop:rep}, 
we know the number of times the characters $*,D$ occurs in the candidate words as well as the possible value of $|P|$. 

Now, we verify that $f$ satisfy the restrictions from Ehrhart theory listed on \cite{class}.

The next step is to choose representatives of Doppelg\"{a}ngers. We restrict the number of candidate words by using the 2--$\operad$ structure under the rule: Any $n$-corolla counts as $n-1$ uses of $*$. 

After that, in a loop:
\begin{itemize}
    \item We evaluate a candidate word $\{w\}$ only on the terms $(f_{i_0},i_0,f_{i_j},i_j)$ of $f$ using Algorithm~\ref{alg:test}.
\item For $w$ that passed the algorithm we compute $w(\zetaf[1])$. This can be done evaluating from top to bottom as in \cite{algo}.
\item If any poset passes this step, we use \cite{survey} to find the remaining Doppelg\"{a}nger posets and we stop the algorithm. Otherwise we continue the loop until there are no more posets candidates.
\end{itemize}

\begin{algorithm}[H]
\SetAlgoLined
\KwIn{ A candidate 2-$\operad$ tree with $i$ leaves and $d$ red  marks. A value $(a_i,i, a_k,k)$ \;}
\tcp{The four numbers stand for $a_i\zetaf[i]+\cdots + a_k\zetaf[k], (1,1,1,1)=\zetaf[1]$}\tcp{ The $n$-corollas represent the composition of $n$ copies of $*$, and the red marks represent $D$}
Put $(1,1,1,1)$ on the leaves; 

Let $B$ be the set of red marks ordered by inverse Deep First Search on the tree\;
 \For{$b\in B$}{compute the input of $b$ with Equation~\eqref{eqn:ex1} and Equation~\eqref{eqn:ex2}. Lets say $(r_b, t_b,s_b,u_b)$\;\tcp{only compute the lowest and highest term}
\uIf{$t_b$ does not divide $a_i$}{
\Return False\;} \Else{store $t_b$ in Cache; \tcp{ Use Cache to find the root's value} }
 }
\uIf{
 the root has a value different to $a_i$
}{\Return False}
\Else{True.}
\caption{Is  $w(1,1,1,1) ==(a_i,i, a_k,k)$? }\label{alg:test}
\end{algorithm}
\end{proof}

The algorithm has as input a description of $f(x)$ as a finite combination of elements $\zetaf[i]$ for some natural numbers $i$. Note that every coefficient of $\sum_{i=1}^\infty \zetaf[i]=x+3x^2+\ldots$ is finite. Because of this, we require the input to be written already as a finite sum of terms $\{x^i/(1-x)^{i+1}\}$. We do not have an algorithm to discover the factorization of an arbitrary power series $f$ that is guaranteed to finish in a finite time.

For arbitrary series-parallel posets we need to adapt Proposition \ref{prop:rep} and increase the search space as now we are required to work with different connected components and attached handles (as in the operation $D$) can consist of several elements.

\begin{proposition}\label{lemma:gen}
Let $Q$ be a series-parallel poset. If $\zetaf[Q]=\sum_{j=i}^{k} a_j\zetaf[j]$ is a finite sum, then 
\begin{enumerate}
\item  the Betti number of $Q$ is smaller or equal to $k-i$,
\item the number $i$ is the length of any maximal chain of the largest connected component of the Hasse diagram of $Q$,
\item the number $k$ is the total number of points in the poset,
\item $a_i$ and $a_k$ can be factored into a product of terms obtained by the application of disjoint union on the subposets of $Q$,
\item $\sum_{u=1}^{k}(-1)^{k-u}a_u=1$.
\end{enumerate}
\end{proposition}

\begin{proof}{\ }
\begin{enumerate}
\item  Geometrically, when we apply $D$ to a disjoint union of posets, we increase the Betti number of the Hasse diagram. But according to Proposition~\eqref{prop:closed1}, the index $k$ reflects not only the number of handles in the Hasse diagram but also the number of points on those handles and the number of connected components on the Hasse diagram. If \begin{equation*}
\sum_{j_1=i_1}^{k_1} a_{j_1}\zetaf[j_1]\sqcup \sum_{j_2=i_2}^{k_2} a_{j_2}\zetaf[j_2]=\sum_{j_3=i_3}^{k_3} a_{j_3}\zetaf[j_3],
\end{equation*}
then we can verify that $k_3-i_3\geq (k_1-i_1)+(k_2-i_2)$.
Here $(k_1-i_1)$ and $(k_2-i_2)$ are bounds for the Betti numbers of each connected component, and $(k_3-i_3)$ bounds the Betti number of the union. Hence the inequality.
    
\item  We use the following observation: If $m\geq k$, then, by the definition of $\sqcup$ (see also Proposition~\ref{prop:closed1}), we find that
\begin{equation*}
\zetaf[k]\sqcup\zetaf[m]={m\choose k}{k\choose 0}\zetaf[m]+\cdots.    
\end{equation*}
The proof is similar to the proof of item~\eqref{prop:point} of Proposition~\ref{prop:rep}.
\item  If $m>n$, then by Proposition~\ref{prop:closed1} the last non-zero coefficient of $\chain[m]\bigsqcup \chain[n]$ is ${m+n\choose n}{n\choose n}\zetaf[m+n]$. Concatenation also preserves the number of elements. These results imply that the term $\zetaf[k]$ is obtained by adding all elements involved in all disjoint unions as well as those elements involved in concatenation.
\item It follows from Proposition~\ref{prop:closed1}.
\item Follows from Equation \eqref{eqn:Ehrhart} and the fact that $h_0^*=1$.
\end{enumerate}
\end{proof}

The coefficients of the power series associated to a series-parallel poset are not topological invariants. For example the posets $\chain[1]*(\chain[1]\sqcup \chain[1])*\chain[1]$ and $\chain[2]*(\chain[1]\sqcup \chain[1])$ have the same number of labeling maps, but different Betti numbers.
  
To adapt Algorithm~\ref{alg:test} for series-parallel posets we need to increase the search space to allow for disjoint union of posets. The algorithm depends on \cite{survey} which already works at the level of series-parallel posets.  

\subsection{Comparison with Ehrhart theory}
\label{subsec:EhrhartTheory}
Let $Poly(P)$ denote the order polytope of a poset $P$ and $E_{Poly(P)}$ its Ehrhart polynomial~\cite{computingd}. Recall that by
 \begin{equation*} \label{eqn:Ehrhart}    \zetafp[P] = \sum_{n=1}^\infty \Omega_{P,n} x^{n}  
 = x+\sum_{n=2}^\infty E_{Poly(P)}(n-1) x^n
     =x \frac{h^*(x)}{(1-x)^{|P|+1}}
 \end{equation*}
order series are related to the Ehrhart series of $Poly(P)$, or its $h^*$-vector, by a change of basis and a shift of degree. 
From the point of view of Ehrhart theory there are several works characterizing $h^*$-vectors \cite{coef,class,monotonicity} and $f^*$-vectors \cite{fvect} of polytopes. We instead study Ehrhart series of order polytopes. This is the reason why we work with respect to a different basis. 

\begin{remark}\label{rem:from_poly_to_order}
Note that 
\begin{equation*}
Poly(P\sqcup Q)=Poly(P)\times Poly(Q), \quad  Poly(P*Q)=Poly(P)\oplus Poly(Q),
\end{equation*}
where $\oplus$ denotes the \textit{free sum} of polytopes: 
If $X \subset \mathbb{R}^d$ and $Y \subset \mathbb{R}^e$ are two polytopes of dimension $d$ and $e$, then $ X \oplus Y =  \mathrm{conv}\{ (0_{\mathbb{R}^d}\times Y) \cup (X \times 0_{\mathbb{R}^e} ) \} \subset \mathbb{R}^{d+e}$. 

This and the relation $\Omega_{P,n} = E_{Poly(P)}(n-1)$ imply the form of the respective operations on order series, Equation~\eqref{hom+a} and Equation~\eqref{hom+u}. See for instance \cite{refl}.
\end{remark} 

We have shown above how to express $\zetafp[P]$ in the basis $\{\zetafp[i]\}_{i\in \mathbb{N}}$.  In this section we relate the coordinates with respect to this basis to a certain triangulation of $Poly(P)$.

Following Stanley \cite{tpp}, a \emph{lower set (order ideal)} $I$ is a subset of the poset $P$ so that if $x\in I$ and $y\leq x$ then $y\in I$. Lower sets form a poset $J(P)$ ordered by inclusion.
To any chain of strict inclusions of lower sets $K:I_1\subset I_2\subset\cdots\subset I_k$ we assign \begin{equation*}F_K= \left\{
f\in \mathbb{R}^{|P|}:\begin{array}{l} f  \hbox{ is constant on the subsets } I_1, I_2\setminus I_1,\cdots I_k\setminus I_{k-1}, P\setminus I_k,\\ 0 =f(I_1)\leq f(I_2)\leq\cdots\leq f(I_k)\leq f(P\setminus I_k)=1.
\end{array}  \right\}.
\end{equation*}
The \emph{canonical triangulation} of the order polytope of $P$ is given by the simplicial complex 
\begin{equation*}
\{F_K\mid K \in J(P)\}.
\end{equation*}
In other words, it is the order complex of the poset $J(P)$ (see \cite{tpp}).

For an example consider the poset $\chain[2]\sqcup\chain[3]$. Its order polytope is 
\begin{equation*}
Poly(\chain[2]\sqcup\chain[3])=\{ x_1 \leq x_2 , y_1\leq y_2 \leq y_3\}\subset [0,1]^5.
\end{equation*}
Using coordinates $(x_1, x_2, y_1, y_2, y_3)$ in $\mathbb{R}^5$. It contains various 5-dimensional simplices, for example $\{0\leq x_1\leq x_2\leq y_1\leq y_2\leq y_3\leq 1\}$, or $\{0\leq x_1\leq y_1\leq x_2\leq y_2\leq y_3\leq 1\}$. This fact can be expressed using a shuffle product (see e.g.\ \cite{shuffle}) on the order polytopes of two chains $\chain[n]$ and $\chain[m]$. We define it as the sum of simplices that are formed by linear orders on $x_1,x_2,\ldots, x_n,y_1,y_2,\ldots,y_m$ that preserve the original order on the $x$- and the $y$-variables. Note that this is just the usual shuffle product -- we simply use words to encode linear orders, for example $xy$ means $x\leq y$, and vice versa. It is known that the shuffle product of two words on $n$ and $m$ letters has ${n+m\choose n}$ summands. This implies that the number of maximal simplices in the canonical triangulation of $Poly(\chain[n]\sqcup\chain[m])$ is ${n+m\choose n}$. In the example $\chain[2]\sqcup\chain[3]$ we find thus 10 simplices of dimension 5 in the canonical triangulation of $Poly(\chain[2]\sqcup\chain[3])$. 
 
The set of simplices $\{F_K\}_{ K \in J(P)}$ is  closed under taking intersections. In particular, every simplex not in the boundary of $Poly(P)$ lies in the intersection of some maximal simplices of the canonical triangulation. Suppose $P_1,\ldots,P_k$ are the maximal simplices of $\{F_K\}_{ K \in J(P)}$. Let ${\mathbb{1}}_{X}\colon \mathbb{R}^d \to \mathbb{R}$ be the indicator function of polytope $X\subset \mathbb{R}^d$, defined by ${\mathbb{1}}_{X}(x)=1$ if $x\in X$ and $0$ otherwise.  We can determine all points on the canonical triangulation of $Poly(P)$ by considering the union of simplices of maximal dimension, but then we are overcounting. For instance, we count twice the faces obtained as intersection of two simplices. We can express this fact as
\begin{equation*}
{\mathbb{1}}_{Poly(P)}=\sum_i {\mathbb{1}}_{P_i}-\sum_{i,j, i\neq j} {\mathbb{1}}_{P_i\cap P_j}+\cdots+(-1)^k{\mathbb{1}}_{P_1\cap\cdots\cap P_k},
\end{equation*}
which is an instance of the well-known inclusion-exclusion principle.

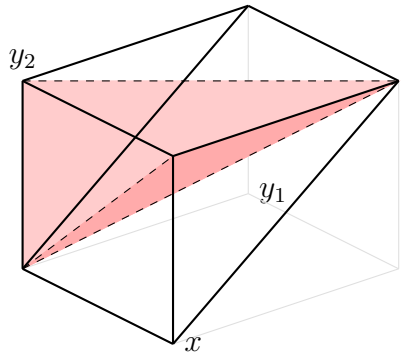
\begin{figure}[h]
\begin{tikzpicture}[scale=1]
\coordinate (v0) at (0,0);
\coordinate  (x1) at (2,-1);
\coordinate  (y1) at (3,1);
\coordinate  (y2) at (0,2.5);
\coordinate  (xy1) at (5,0);
\coordinate  (xy1y2) at (5,2.5);
\coordinate  (y1y2) at (3,3.5);
\coordinate  (xy2) at (2,1.5); 
\draw[lightgray!50] (v0) to (y1) node[black,right]{$y_1$};
\draw[lightgray!50] (x1) to (xy1);
\draw[lightgray!50] (xy1) to (xy1y2);
\draw[lightgray!50] (y1) to (y1y2);
\draw[lightgray!50] (y1) to (xy1); 
\filldraw[red!20] (v0) -- (xy1y2) -- (y2) -- (v0);
\filldraw[red!33] (v0) -- (xy2) -- (xy1y2) -- (v0);
\draw[thick] (v0) to (x1) node[right]{$x$};
\draw[thick] (v0) to (y2) node[above]{$y_2$};
\draw[thick] (x1) to (xy1y2);
\draw[thick] (v0) to (y1y2);
\draw[thick] (y1y2) to (xy1y2);
\draw[thick] (y2) to (xy2);
\draw[thick] (x1) to (xy2); 
\draw[thick] (xy2) to (xy1y2);
\draw[thick] (y2) to (y1y2);
\draw[dashed] (v0) to (xy1y2); 
\draw[dashed] (y2) to (xy1y2);
\draw[dashed] (v0) to (xy2); 
\end{tikzpicture}
\caption{The order polytope of the poset $\chain[1]\sqcup \chain[2]$. The canonical triangulation consists of three 3-simplices, indicated by the dashed lines. They intersect in two internal 2-simplices, depicted in red.}\label{fig:polyP}
\end{figure}

\begin{lemma}\label{lemma:aux}
Let $m\geq k\geq r$, $P=\chain[k]\sqcup\chain[m]$. Then the number of $m+r$--dimensional simplices in the canonical triangulation of $P$ is ${m+r\choose k}{k\choose r}$.
\end{lemma}
\begin{proof}
 
Consider a $(m+r)$--simplex $T$ in the canonical triangulation of $Poly(P)$. We have already discussed above that in the case $r=k$ the simplex $T$ is represented by a linear order on the coordinates of $\mathbb{R}^{|P|}=\mathbb{R}^k \times \mathbb{R}^m = \text{span}_{\mathbb{R}}(x_1,\ldots,x_k,y_1,\ldots,y_m)$ that respects the individual orders $x_1\leq\ldots\leq x_k$ and $y_1 \leq \ldots \leq y_m$. If $r<k$, then $T$ is obtained from one (or more) of these simplices by exchanging an inequality $``\leq"$ for an equality $``="$. For instance, if $k=m=1$, then the canonical triangulation of $Poly(P)=[0,1]^2$ has two 2-simplices $\{0\leq x \leq y \leq 1\}$ and $\{0\leq y\leq x \leq 1\}$ which intersect in the internal 1-simplex $\{0\leq x=y \leq 1\}$. See also Example \ref{example:cantriang} below. 

The general construction starts with a geometrical $m+r$-simplex $\Delta_{m+r}=\{(z_0,\cdots,z_{m+r-1})|$ $\sum z_i=1, z_i\geq 0\}$. We then change coordinates $w_1=z_0, w_2=z_0+z_1,\cdots, w_{m+r}=\sum z_i  $, to obtain
$\Delta_{m+r}=\{0\leq w_1\leq \cdots\leq w_{m+r}\leq 1\}$. The final step is to embed this simplex into the canonical triangulation by assigning coordinates $x_i, y_j$ to the $w_k$ place holders. 

Choose $k$ different coordinates of $\Delta_{m+r}$; this can be done in ${m+r\choose k}$ ways.
Out of those $k$ coordinates choose $k-r$ coordinates in ${k\choose r}$ ways. Then label the $(m+r)-k+(k-r)=m$ coordinates from $x_1$ to $x_m$. Add a second label $y_1,\cdots y_k$ to the $k$ chosen points. Those $k-r$ chosen points have now two labels $x_i,y_j$ which represent the condition $x_i=y_j$. Any other point has only a single label. The simplex $T$ is then defined by applying these labels to $\Delta_{m+r}$.
We can do this process in ${m+r\choose k}{k\choose r}$ ways, as claimed, and any simplex in the canonical triangulation is determined by such a process.
\end{proof}

\begin{proposition}[$\zetafp$ satisfies the inclusion-exclusion principle]\label{IElemma}
Let $n=|P|$. Consider $\zetafp[{P}]=\sum_{i=j_0}^{n} (-1)^{n-i}a_i\zetafp[i]$. 
Then every $a_i$ is non-negative, the order polytope of ${P}$ is the union of $a_{n}$ copies of $n-$simplices, and the remaining coefficients $a_i$ of $\zetafp[{P}]=\sum (-1)^{n-i}a_i\zetafp[i]$ count the number of internal $i-$faces that occur as intersections of the $n$-simplices.
\label{lemma:a}
\end{proposition}

\begin{proof}
Case $P=\chain[k]* \chain[m]$:
A $k$--chain concatenated with a $m$--chain is a $(m+k)$-chain, so the coefficient of $\zetafp[m+k]$ is 1, in accordance with the fact that $Poly(\chain[m+k])$ is a $(m+k)$--dimensional simplex.

Case $P=\chain[k]\sqcup \chain[m]$: Assume $m\geq k$. Lemma~\ref{lemma:aux} together with Proposition~\ref{prop:closed1} and Remark~\ref{remark:reciprocity:coef}, implies that the inclusion-exclusion principle is valid for this case.

Now, if (the non-strict order series of) a poset $P$ satisfies the inclusion-exclusion formula, so does $P*\chain[1]$ and $P\sqcup \chain[1]$.  The geometric reason is that both operations add new orthogonal directions to a unit cube. The action of $*\chain[1]$ is that of forming a cone while $\sqcup\chain[1]$ acts as the Minkowski sum; the order polytope $Poly(P\sqcup Q)$ is constructed from the order polytopes $Poly(P)$ and $Poly(Q)$ by choosing orthogonal dimensions with coordinates $\{x_i\}$ and $\{y_j\}$, respectively.  On each we draw one order polytope, and then we consider the space $ax+yb$, $x\in Poly(P), y\in Poly(Q)$ with $0\leq a,b\leq 1$.

For any $s$--simplex in $Poly(Q)$ there is a chain $S$ in $J(P)$ so that the $s$--simplex is the order polytope of $S$.  
 
Then $Poly(P\sqcup Q)= \sqcup_{S\in Poly(Q)} Poly(P\sqcup S)$. Since $*,\sqcup$ are distributive with respect to $+$ on the vector space with basis $\{\zetafp[{i}]\}$, assuming $\zetaf[Q]=\sum a_i\zetaf[i]$, we obtain $\zetaf[P\sqcup Q]=\sum a_i\zetaf[P]\sqcup \zetaf[i]$. Since we proved that for chains the inclusion-exclusion formula holds, it is thus also true for series-parallel posets.
\end{proof}

\begin{example}\label{example:cantriang}
Consider $P=\chain[1] \sqcup \chain[2]$. Figure \ref{fig:polyP} depicts the order polytope of $P$. The canonical triangulation has three 3-simplices, described by the linear extensions $
\{x\leq y_1\leq y_2\}$ , $\{y_1\leq x \leq y_2 \}$ and $\{y_1 \leq y_2 \leq x\}$.
If $r=0$, then ${m+r\choose k}{k\choose r}=2$, so there are two internal 2-simplices which arise as intersections of the latter ones. The first two produce $\{x=y_1 \leq y_2 \}$, the last two produce $\{y_1 \leq y_2 = x\}$. Note that there are no internal 1-simplices; every 1-simplex of the canonical triangulation lies in the boundary of $Poly(P)$.
\end{example}
 
It is known that Ehrhart polynomials satisfy the inclusion-exclusion principle since they count lattice points on integer polytopes. In \cite[Proposition~9.5]{hopf} the Hopf monoid structure of posets is studied. It is proved that posets are isomorphic, as Hopf monoids, to pointed conical generalized permutahedra. In this sense, the order polynomial is the \textit{associated polynomial to a character} \cite[Proposition~9.4]{hopf}.
It follows that the order polynomial satisfies an inclusion-exclusion formula for pointed conical generalized permutahedra. In our setting, we consider the action of an operad generated by concatenation and disjoint union. 
In that case Proposition~\ref{IElemma} shows that the inclusion-exclusion principle is preserved by the action of these two operations.

\section{Combinatorics}
\label{sec:comb}
In the first subsection we apply our results to study binomial identities. We could not find the identities in Equation~\eqref{eq:nodiv}, Proposition~\ref{Proposition:pc} and Proposition~\ref{prop:new} in the book \cite{lcomb}. We explore some corollaries of those identities, including an identity for finite partitions of fixed length that are allowed to contain empty sets (Subsection \ref{subsec:partitions}).
 In Section~\ref{sec:prob} we introduce a combinatorial approach to study the multivariate negative hypergeometric distribution. 
 
\subsection{Identities}
The existence of the operad structure on order series implies Proposition~\ref{Proposition:pc}. The associativity of the concatenation on order series implies Proposition~\ref{prop:new}. We call these two identities structural binomial identities. We also list several direct consequences of these structural identitites that we could not find in the literature.

\begin{proposition}\label{Proposition:pc}
Given $p,q,s\leq p+q$. Then for $t\leq s$:

\begin{align*}
 {p+q+t\choose s}{s\choose t}=  \sum_{\substack{a+c=s\\
n\leq a\\
r\leq c\\
n+r=t}}{p+n \choose a}{a \choose n}{q+r \choose c}{c \choose r}-\sum_{\substack{a+c=s-1\\n\leq a\\
r\leq c\\
n+r=t-1}}
{p+n \choose a}{a \choose n}{q+r \choose c}{c \choose r}.
\end{align*}

\end{proposition}

\begin{proof}
Consider Equation \eqref{compatibility} with $f=\zetaf[q], g=\zetaf[p]$. Then 
\begin{align*}
\zetaf[s]\sqcup (\zetaf[q]*\zetaf[p])&=\zetaf[s]\sqcup (\zetaf[q+p])\\
&=\sum_{t=0}^s{p+q+t \choose s}{s \choose t}\zetaf[p+q+t],
\end{align*}
and
\begin{align*}
\zetaf[s]\sqcup (\zetaf[q]*\zetaf[p])=&
\sum_{a+c=s} (\zetaf[a]\sqcup \zetaf[p])*(\zetaf[c]\sqcup \zetaf[q])\\
&-\left(\sum_{a+c=s-1} (\zetaf[a]\sqcup \zetaf[p])*(\zetaf[c]\sqcup \zetaf[q])\right) *\zetaf[1]\\
=&\sum_{a+c=s}\sum_{n=0}^a{p+n \choose a}{a \choose n}\zetaf[p+n]
*\sum_{r=0}^c{q+r \choose c}{c \choose r}\zetaf[q+r]\\
&-\sum_{a+c=s-1}\sum_{n=0}^a{p+n \choose a}{a \choose n}\zetaf[p+n]
*\sum_{r=0}^c{q+r \choose c}{c \choose r}\zetaf[q+r]*\zetaf[1]\\
=&\sum_{a+c=s}\sum_{n=0}^a\sum_{r=0}^c{p+n \choose a}{a \choose n}{q+r \choose c}{c \choose r}\zetaf[p+n+q+r]\\
&-\sum_{a+c=s-1}\sum_{n=0}^a\sum_{r=0}^c
{p+n \choose a}{a \choose n}{q+r \choose c}{c \choose r}\zetaf[p+n+q+r+1].
\end{align*}
 The result follows by comparing the coefficients.
\end{proof}
Note that the previous proposition generalizes the Chu-Vandermonde identity. To see this, evaluate on $t=0$.

\begin{proposition}\label{prop:new}
Let $n$ be a non-negative integer and  $ k\in \mathbb{N}$. For any $k$-partition $n=n_1+\cdots+n_k$, with $n_i\geq 0$, and for $m\geq n+k-1$ we have
\begin{multline*}
{m \choose n}={k-1\choose 0}\sum_{\substack{\sum m_i=m\\ m_i\geq n_i}}\prod_{i=1}^k {q_i\choose n_i} -{k-1\choose 1}\sum_{\substack{\sum m_i=m-1\\ m_i\geq n_i}}\prod_{i=1}^k {q_i\choose n_i}\pm\cdots\\
 +(-1)^{k-2}{k-1\choose k-2}\sum_{\substack{\sum m_i=m-(k-2)\\ m_i\geq n_i}}\prod_{i=1}^k {q_i\choose n_i}+   (-1)^{k-1}{k-1\choose k-1}\sum_{\substack{\sum m_i=m-(k-1)\\ m_i\geq n_i}}\prod_{i=1}^k {q_i\choose n_i}.
\end{multline*}

\end{proposition}
\begin{proof}
Decompose $\chain[n]$ as $\chain [n_1]*\chain[n_2]*\cdots *\chain[n_k]$. Then the result is obtained by applying Equation~\eqref{eqn:concat} and computing the degree $m$ coefficient of the last row in the sequence of identities below:
\begin{align*}
\sum_{q=n}^\infty {q\choose n} x^m &= 
\zetaf[n]\\
&=\zetaf[n_1](1-x)\zetaf[n_2](1-x)\cdots (1-x)\zetaf[n_k]\\
 &=\zetaf[n_1]\zetaf[n_2]\cdots \zetaf[n_k](1-x)^{k-1}\\
&=\sum_{q_1=n_1}^\infty {{q_1}\choose n_1} x^{q_1} \cdots \sum_{q_{k}=n_k}^\infty {{q_{k}}\choose n_k} x^{q_{k}} \sum_{j=0}^{k-1} (-1)^j{k-1\choose j}x^j.
\end{align*}
\end{proof}
Using $\zetafp[n]$ and $*^+$  we obtain the equivalent identity:
\begin{proposition}
\label{eqn:gen}
Let $n=n_1+\cdots+n_k$. 
Then for $v\geq n+k-1$
\begin{multline*}
 \multiset{v-k+1}{n} ={k-1\choose 0}\sum_{\substack{\sum v_i=v\\ v_i\geq 1}}
 \prod_{i=1}^k \multiset{v_i}{n_i} -{k-1\choose 1}\sum_{\substack{\sum v_i=v-1\\ v_i\geq 1}}
 \prod_{i=1}^k \multiset{v_i}{n_i}+\cdots
 \\
 +(-1)^{k-2}{k-1\choose k-2}\sum_{\substack{\sum v_i=v-(k-2)\\ v_i\geq 1}}
 \prod_{i=1}^k \multiset{v_i}{n_i} +   {k-1\choose k-1}\sum_{\substack{\sum v_i=v-(k-1)\\ v_i\geq 1}}
 \prod_{i=1}^k \multiset{v_i}{n_i}.
\end{multline*}
\end{proposition}

Following the ideas of the previous section, we define  $f(x)*^-g(x)= f(x)\frac{1}{x} g(x)$ and generate the monoid over $\frac{x}{(1-x)^2}$. In this case we study the combinatorics of the product of Ehrhart series (with respect to the basis formed by chains), see \cite[Problem 6.3]{LN}. We associate the power series $\zetafm[1]=\frac{x}{(1-x)^2}$ to the chain $\chain[1]$ and extend by $\zetafm[n]=\zetafm[n-1]*^-\zetafm[1]$ to get $\zetafm[n]=\frac{x}{(1-x)^{2n}}$. Applications of this structure to probability theory will be explained in Section~\ref{sec:prob}.

Using the expansion of the Maclaurin series, we obtain the Vandermonde identity for negative integers:
\begin{proposition}
Let $n=n_1+\cdots+n_k$. 
Then for $v\geq k$ :\label{prop:gen3}
\begin{equation}\label{eqn:gen3}
 \multiset{v-k+1}{2n-1} =\sum_{\substack{\sum v_i=v\\ v_i\geq 1}}
 \prod_{i=1}^k \multiset{v_i}{2n_i-1}.
\end{equation}
\end{proposition}

\begin{proof}\label{Cor:41}
It follows from $\zetafm[n]=\zetafm[n_1]*^-\zetafm[n_2]\cdots\zetafm[n_k]$.
\end{proof}

From Equation~\eqref{eqn:gen3} and $n=1+\cdots+1$ we obtain for $v\geq n$
\begin{equation}\label{eq:noprod}
 {v+n-1\choose 2n-1} =\sum_{\substack{\sum v_i=v\\ v_i\geq 1}}
   \prod_{i=1}^n {v_i}.
\end{equation}
This gives a formula to compute certain binomial coefficients without using any division. 
\begin{proposition}
Given $q \in \mathbb{N}$, then for $n\leq\frac{m+1}{2}$ we have \begin{equation}\label{eq:nodiv}
  {q\choose 2n-1}=\sum_{\substack{\sum v_i=q-(n-1)\\ v_i\geq 1}}
 \prod_{i=1}^n {v_i}. 
\end{equation}

\end{proposition}
\begin{proof}
 It follows from Equation \eqref{eq:noprod}.
\end{proof}

The latter identity requires that the lower entry in the binomial coefficient must be odd. Consider instead $n=1+\cdots+1$. Then using Proposition \ref{prop:new}, for $m\geq 2n-1$,

\begin{multline}\label{eq:nodiv2}
{m \choose n}={n-1\choose 0}\sum_{\substack{\sum m_i=m\\ m_i\geq 1}}\prod_{i=1}^n m_i -{n-1\choose 1}\sum_{\substack{\sum m_i=m-1\\ m_i\geq 1}}\prod_{i=1}^n m_i\pm\cdots
\\
+(-1)^{n-2}{n-1\choose n-2}\sum_{\substack{\sum m_i=m-(n-2)\\ m_i\geq 1}}\prod_{i=1}^n {q_i}+   (-1)^{n-1}{n-1\choose n-1}\sum_{\substack{\sum m_i=m-(n-1)\\ m_i\geq 1}}\prod_{i=1}^n {q_i}. 
\end{multline}

If we use $n^2=n+\cdots+n$, then using  Proposition \ref{prop:new},  for $m\geq n^2+n-1$,
\begin{equation*}
{m \choose n^2}={n-1\choose 0}\sum_{\substack{\sum m_i=m\\ m_i\geq n}}\prod_{i=1}^n {q_i\choose n} -{n-1\choose 1}\sum_{\substack{\sum m_i=m-1\\ m_i\geq n}}\prod_{i=1}^n {q_i\choose n}\pm\cdots
\end{equation*}
\begin{equation}
+(-1)^{n-2}{n-1\choose n-2}\sum_{\substack{\sum m_i=m-(n-2)\\ m_i\geq n}}\prod_{i=1}^n {q_i\choose n}+   (-1)^{n-1}{n-1\choose n-1}\sum_{\substack{\sum m_i=m-(n-1)\\ m_i\geq n}}\prod_{i=1}^n {q_i\choose n}.
\end{equation}

Similar identities allow us to compute ${m \choose n^k}$ in terms of ${m \choose n^r}$ for $r<k$.

On Equation~\eqref{eq:noprod}, for every partition $\sum v_i=v$, if  $v_i$ appears $s_i$ times then the term $\prod_i v_i$ will appear as many times as the \emph{multinomial coefficient} ${n \choose s_1,\cdots, s_j}$. We compute

\begin{equation}\label{eq:short}
 {v+n-1\choose 2n-1} =\sum_
 {\substack{ \sum v_i=v\\ \#v_i=s_i }}
 {n \choose s_1,\cdots, s_j}\prod_{i=1}^n {v_i}, 
\end{equation}
where the sum is over different $n$-partitions of $v$, and each  term $v_i$ appears $s_i$ times.  

Finally, we apply Proposition \ref{eqn:gen} to the case $2n-k=\sum^k_{i=1} 2n_i-1, v\geq 2n-1$.  
\begin{multline*}
 \multiset{v-k+1}{2n-k} ={k-1\choose 0}\sum_{\substack{\sum v_i=v\\ v_i\geq 1}}
 \prod_{i=1}^k \multiset{v_i}{2n_i-1} -{k-1\choose 1}\sum_{\substack{\sum v_i=v-1\\ v_i\geq 1}}
 \prod_{i=1}^k \multiset{v_i}{2n_i-1}\pm
 \\
 \cdots+(-1)^{k-2}{k-1\choose k-2}\sum_{\substack{\sum v_i=v-(k-2)\\ v_i\geq 1}}
 \prod_{i=1}^k \multiset{v_i}{2n_i-1} +   {k-1\choose k-1}\sum_{\substack{\sum v_i=v-(k-1)\\ v_i\geq 1}}
 \prod_{i=1}^k \multiset{v_i}{2n_i-1}.
\end{multline*}

We then substitute Equation~\eqref{eqn:gen3} to obtain, for $ v\geq 2n-1+k-1$,
\begin{multline*}
\multiset{v-k+1}{2n-k} ={k-1\choose 0}\multiset{v-k+1}{2n-1} -{k-1\choose 1}\multiset{v-k}{2n-1}\pm\cdots
\\
+(-1)^{k-2}{k-1\choose k-2}\multiset{v-2k+3}{2n-1} +   (-1)^{k-1}{k-1\choose k-1}\multiset{v-2k+2}{2n-1},
\end{multline*}
or 
\begin{multline*}
 {v-2k+2n\choose 2n-k} ={k-1\choose 0}{v-k-1+2n\choose 2n-1} -{k-1\choose 1}{v-k-2+2n\choose 2n-1}\pm\cdots\\
 +(-1)^{k-2}{k-1\choose k-2}{v-2k+1+2n\choose 2n-1} +   (-1)^{k-1}{k-1\choose k-1}{v-2k+2n\choose 2n-1}.
\end{multline*}

Our identities come from objects in which a notion of  series-parallel order is defined. It is expected that new series-parallel algebras will imply new properties of binomial coefficients.

\subsection{Finite partitions that allow empty sets}\label{subsec:partitions} Define  $\tilde{n}(m,k)=\sum_{\sum_{i=1}^k m_i=m}1$. In words, $\tilde{n}(m,k)$  is the number of partitions of $m$ points into $k$ sets that can be empty. 
 Let $n=0$, $k\in \mathbb{N}$, $m\geq k-1$. From Proposition~\ref{prop:new} we obtain the identity
\begin{multline}
1={k-1\choose 0}\Tilde{n}(m,k) -{k-1\choose 1}\Tilde{n}(m-1,k)+\cdots\\
     +(-1)^{k-2}{k-1\choose k-2}\Tilde{n}(m-(k-2),k)+   (-1)^{k-1}{k-1\choose k-1}\Tilde{n}(m-(k-1),k).
\end{multline}    
The Stirling number of the second kind ${m \brace k}$ is defined as the number of partitions of a set with $m$ elements into $k$ non-empty subsets. 
It follows that
\begin{equation*}\Tilde{n}(m,k) = {k\choose 0}{m \brace k} + {k\choose 1} {m \brace k-1}+\cdots+{k\choose k-1} {m \brace 1}.\end{equation*}

\subsection{Probability}\label{sec:prob}

Consider $N$ classes of objects. Take $V$ objects with repetition. Divide the $N$ classes into groups, the first with $n_1$ classes, the next with $n_2$ classes, and so on. We assume $N=n_1+\cdots+n_k$. Let $V=v_1+\cdots+v_k$. What is the probability that $v_1$ objects belong to the group of $n_1$ fixed classes, and $v_2$ objects belong to the group of $n_2$ fixed classes, and so on? The answer comes from the multivariate negative hypergeometric distribution, 
 \begin{equation*} P(X=(v_1,\cdots,v_k))= \frac{\multiset{n_1}{v_1}\dots\multiset{n_k}{v_k} }{\multiset{N}{V}}.\end{equation*}
As an application of our methods we provide a direct proof that $P$ is a distribution.
 
\begin{proposition}
The function $P$ is a distribution.
\end{proposition}
\begin{proof}
 From $\frac{x}{(1-x)^{N}}=\frac{x}{(1-x)^{n_1}}\frac{1}{x}\cdots\frac{1}{x}\frac{x}{(1-x)^{n_k}}$ and Proposition~\ref{prop:gen3} we obtain \begin{equation*}\multiset{V-(k-1)}{N-1} =\sum_{\substack{\sum v_i=V\\ v_i\geq 1}}
 \prod_{i=1}^k \multiset{v_i}{n_i-1}.\end{equation*}
 Equivalently,
\begin{equation*}\multiset{N}{V-k} =\sum_{\substack{\sum v_i=V\\ v_i\geq 1}}
 \prod_{i=1}^k \multiset{n_i}{v_i-1},\end{equation*}
 and by changing variables we have
 \begin{equation*}\multiset{N}{W} =\sum_{\substack{\sum w_i=W\\ w_i\geq 0}}
 \prod_{i=1}^k \multiset{n_i}{w_i} \ \text{ and } \ 1 =\sum_{\substack{\sum w_i=W\\ w_i\geq 0}}
 \frac{\prod_{i=1}^k \multiset{n_i}{w_i}}{\multiset{N}{W}}.\end{equation*}
\end{proof}
We also provide a method to compute the expectation value of the negative hypergeometric distribution:
\begin{eqnarray*}
\sum_{W=1} x^W\sum_{w=1}^W w\multiset{n}{w}\multiset{N-n}{W-w}&=&\left(\sum_{i=1} i\multiset{n}{i}x^i\right)\left(\sum_{j=1} \multiset{N-n}{j}x^j\right)
\\ 
&=&\left(x\frac{d}{dx}\left(\frac{1}{(1-x)^{n}}\right)\right)
 \frac{1}{(1-x)^{N-n}}\\
 &=&\frac{nx}{(1-x)^{N+1}}\\
 &=&\sum_{W=1} n\multiset{N+1}{W-1}x^W.\end{eqnarray*}
It follows that \begin{equation*}\mathbb{E}(P)=\sum_{w=0}^W \frac{w\multiset{n}{2}\multiset{N-n}{W-w}}{\multiset{N}{W}}=\frac{n\multiset{N+1}{W-1}}{\multiset{N}{W}}.\end{equation*}

\section{Open problems}
We conclude with a list of open problems.

\begin{itemize}\label{questions}
 \item Do the polynomials in Remark \ref{rem:moebius} have a combinatorial interpretation? 
    \item Is there a combinatorial proof of Proposition~\ref{Proposition:pc}?
    \item Is there a proof of Proposition~\ref{prop:new} by computing the Euler characteristic or using the inclusion-exclusion principle?
    \item What is the topological story behind the $*^-$ structure linked to the negative hypergeometric distribution?
    \item What are the generators of the operad of finite posets?
\end{itemize}
\section*{Acknowledgements}
We thank Elton P. Hsu for carefully reading the section on probability. We thank Max Hopkings, Isaias Marin Gaviria, and Mario Sanchez for several discussions.  We thank John Baez for online discussions about binomial identities and Theo Johnson-Freyd for discussions about the units of the operad of series-parallel posets. Furthermore, we thank Raman Sanyal for suggesting the reference \cite{crt}, and the anonymous referee(s) for many helpful comments and for pointing us to the papers \cite{algo}, \cite{beginning} and \cite{prior}. The second author was funded through the Royal Society grant URF$\setminus$R1$\setminus$20147. The third author was funded by the National Research Foundation of Korea (NRF) grant funded by the Korea government (MSIT) (No. 2020R1C1C1A0100826).
 This project started while the third author worked at NewSci Labs and he thanks them for their hospitality.
\bibliographystyle{plain}
\bibliography{references}

\end{document}